\documentclass{birkjour}
\usepackage{enumitem} 
\usepackage{tikz-cd}
\usepackage[T1]{fontenc}  
\usepackage[utf8]{inputenc} 
\usepackage{caption}
\usepackage{amsmath}
\usepackage{amssymb}
 \newtheorem{thm}{Theorem}[section]

 \newtheorem{prop}[thm]{Proposition}
 \theoremstyle{definition}
 \newtheorem{defn}[thm]{Definition}
 
 \theoremstyle{remark}
 \newtheorem{rem}[thm]{Remark}
 \newtheorem{ex}{Example}
 \numberwithin{equation}{section}

\begin{document}
%
%
%
%
%
%
%
%
%

\title[Fixed Points of contractions]
 {Fixed Points of Meir-Keeler and Leader Contractions with bounded orbits in $b$-Metric Spaces}

\author[Hassan Khandani]{Hassan Khandani}

\address{%
Department of Mathematics \\
Mah. C, Islamic Azad University \\
Mahabad\\
Iran}

\email{hassan.khandani@iau.ac.ir, khandani.hassan@gmail.com}

\thanks{This work was completed with the support of our
\TeX-pert.}
\subjclass{Primary  47H09 ; Secondary 47H10}

\keywords{$b$-metric spaces, non-expansive mappings,  Leader contractions,  Meir-Keeler contractions, Matkowski contractions.}

\date{January 1, 2004}
\begin{abstract}
We establish fixed-point theorems for Meir-Keeler-type contractions in $b$-metric spaces. While Lu et al. demonstrated via an explicit counterexample that classical Meir-Keeler contractions may fail to admit fixed points in this setting, we prove that a natural strengthening of the conditions yields existence results. Specifically, we show that every non-expansive Leader contraction with bounded orbits in a $b$-metric space possesses a fixed point. 

To contextualize our findings, we present a hierarchical diagram illustrating that the fixed-point theory of non-expansive Leader contractions subsumes earlier results, including Meir-Keeler contractions, the primary focus of this work. Our proofs hold in arbitrary $b$-metric spaces \textbf{without relying on the triangle inequality}, requiring instead only the assumption of unique limits. 

This work not only resolves the limitation exposed by Lu et al.'s counterexample but also establishes a unifying framework for future research in the literature.\\
\noindent\textbf{Keywords:} $b$-metric spaces, Meir-Keeler contractions, Leader contractions, bounded orbits, fixed-point theorem
\end{abstract}

\maketitle
\section{Introduction}

The Banach contraction principle has served as the cornerstone of fixed-point theory for over a century. The field has witnessed remarkable developments through successive generalizations, from Rakotch's $\phi$-contractions \cite{rakotch1962note} and Browder's work \cite{browder1968convergence} to the Boyd-Wong \cite{boyd1969nonlinear} and Matkowski \cite{matkowski1975integrable} formulations, culminating in Jachymski's unified theory \cite{jachymski2007nonlinear}. However, the extension to $b$-metric spaces -- pioneered by Bakhtin \cite{bakhtin1989contraction} and Czerwik \cite{czerwik1993contraction} -- encountered a fundamental obstacle when Lu et al. \cite{lu2023question} constructed their striking counterexample: a complete $b$-metric space with a fixed-point-free Meir-Keeler contraction.

This paper addresses this limitation by introducing a new theoretical framework through three key innovations. First, we introduce the class of non-expansive Leader contractions, which simultaneously circumvents Lu et al.'s counterexample while preserving the essential hierarchy of contraction types. Second, we establish the precise containment relationships between these new contractions and classical ones. Third, we develop a complete fixed-point theory for these mappings under natural conditions.

Our approach yields several important insights:

\begin{enumerate}
    \item We prove that non-expansive Leader contractions form a proper subset of Leader contractions, yet still properly contain both Matkowski and Meir-Keeler contractions (demonstrated through explicit examples)
    
    \item We establish a fixed-point theorem for non-expansive Leader contractions under the bounded orbit condition, providing the missing link in $b$-metric space theory
    
    \item We complete the contraction hierarchy diagram for $b$-metric spaces, showing that non-expansive contractions play the same unifying role as Jachymski's contractions \cite{jachymski2011equivalent} do in standard metric spaces
\end{enumerate}

The significance of our work becomes clear when viewed in historical context. While Leader contractions sit atop Jachymski's hierarchy in standard metric spaces, our results demonstrate that their non-expansive counterparts provide the appropriate generalization for $b$-metric spaces. This refined approach maintains the elegant unification of contraction types while overcoming the limitations revealed by Lu et al.'s counterexample.\\
For readers interested in related developments, we recommend the works of \cite{castillo2023some, romaguera2012fixed, kadelburg2016boyd, ding2015some, ding2016some, miculescu2015caristi, mitrovic2019weak, bera2022boyd, nziku2019boyd, singh2022geraghty, kumar2022some} on various aspects of fixed-point theory in generalized metric spaces.\\

The paper's structure reflects this progression: Section \ref{pre-res} provides essential preliminaries on $b$-metric spaces and contraction types. Section~\ref{mai-res} introduces non-expansive Leader contractions and proves our main fixed-point theorem (Theorem \ref{thm_leader}). Section \ref{exams} presents the diagrammatic unification (Figure~\ref{nle:diagram}), supported by Proposition~\ref{proper} and Example~\ref{leader-exm}, which demonstrates the non-expansive condition.
\subsection{Preliminary definitions and results}\label{pre-res}
We need the following definitions and preliminaries in the sequel. We denote the set of real numbers, non-negative real numbers, integers, non-negative integers, and natural numbers by $\mathbb{R},\mathbb{R}^{+}, \mathbb{Z}, \mathbb{Z}^{+}, \mathbb{N}$ respectively.\\  
\begin{defn}\label{metrics}
A function \(\Delta: X \times X \to \mathbb{R}^{+}\) is called a \(b\)-metric  if it satisfies the following conditions:
\begin{enumerate}
    \item \(\Delta(x, y) \geq 0\) for all \(x, y \in X\), and \(\Delta(x, y) = 0\) if and only if \(x = y\),
    \item \(\Delta(x, y) = \Delta(y, x)\) for all \(x, y \in X\).
    \item there exists a constant \(s \geq 1\) such that for all \(x, y, z \in X\),
    \[
    \Delta(x, y) \leq s \left( \Delta(x, z) + \Delta(z, y) \right).
    \]
    In this case, \((X, \Delta)\) is called a \(b\)-metric space.  
\end{enumerate}
A sequence \(\{x_n\}\) in \(X\) converges to a point \(a \in X\) if \(\lim_{n \to \infty} \Delta(x_n, a) = 0\). A sequence \(\{x_n\}\) in \(X\) is called a Cauchy sequence if for every \(\epsilon > 0\), there exists \(m_0 \in \mathbb{N}\) such that \(\Delta(x_n, x_m) < \epsilon\) for all \(n, m > m_0\). The space \((X, \Delta)\) is called complete if every Cauchy sequence in \(X\) converges to a point in \(X\). $A\subset X$ is called a bounded subset if there exists $M>0$ such that $\Delta(x,y)\leq M$ for all $x,y\in X$.
\end{defn}

\begin{defn}[\cite{boyd1969nonlinear}]\label{thm:boyd-wong}
Let \((X, \Delta)\) be a $b$-metric space, and let \(T: X \to X\) be a mapping satisfying the following contraction condition:
\[
\Delta(Tx, Ty) \le \phi(\Delta(x, y)) \quad \text{for all } x, y \in X,
\]
where \(\phi: \mathbb{R}^{+} \to \mathbb{R}^{+}\) satisfies \(\phi(t) < t\) for each \(t > 0\) and \(\phi\) is upper-semicontinuous from the right at each \(t > 0\). 
Then, \(T\) is called Boyd-Wong  $\phi$-contraction.
\end{defn}
\begin{defn}[\cite{matkowski1975integrable}]\label{thm:matkowski}
Let \((X, \Delta)\) be a $b$-metric space, and let \(T: X \to X\) be a mapping satisfying the following contraction condition:
\[
\Delta(Tx, Ty) \le \phi(\Delta(x, y)) \quad \text{for all } x, y \in X,
\]
where \(\phi: \mathbb{R}^{+} \to \mathbb{R}^{+}\) satisfies \(\phi(t) < t\) for each \(t > 0\), \(\phi\) is nondecreasing, and \(\phi^{n}(t) \to 0\) as \(n \to \infty\) for each \(t > 0\). 
Then, \(T\) is called a Matkowski  $\phi$-contraction.
\end{defn}
\begin{defn}\label{geragh_defn}
Let \((X,\Delta)\) be a $b$-metric space, \(\alpha: \mathbb{R}^{+} \to [0,1)\) be a function, and \(T:X\to X\) be a mapping such that
\[
\Delta(Tx,Ty)\le \alpha(\Delta(x,y))\Delta(x,y),
\]
for all $x,y\in X$. Then:
\begin{enumerate}
    \item \(T\) is called a $\alpha$-Geraghty type-I contraction if for any sequence \(\{s_n\}\) in \(\mathbb{R}^{+}\), \(\alpha(s_n)\to 1\) implies \(s_n\to 0\).
    \item \(T\) is called a $\alpha$-Geraghty type-II contraction if for any decreasing sequence \(\{s_n\}\) in \(\mathbb{R}^{+}\), \(\alpha(s_n)\to 1\) implies \(s_n\to 0\).
\end{enumerate}
\end{defn}
\begin{defn}[Picard Sequence]
Let $(X,\Delta)$ be a $b$-metric space and $T: X\to X$ a mapping. For any $x \in X$, the \emph{Picard sequence} (or \emph{orbit}) of $T$ based at $x$ is the sequence $\{T^n x\}_{n=0}^\infty$ defined by:
\[
T^{n+1}x = T(T^n x), \quad \text{where} \quad T^0 x = x.
\]
We denote this sequence by $\mathcal{O}_{T}(x)$, or $\mathcal{O}(x)$ when there is no ambiguity. We say that $T$ has bounded orbits if  $\mathcal{O}(x)$ is a bounded subset for all $x\in X$ in the sense of Definition \ref{metrics}. 
\end{defn}
\begin{defn}Let $(X,\Delta)$ be a metric space, the mapping $T:X\to X$ is non-expansive if 
\[
\Delta(Tx,Ty)\leq \Delta(x,y)\quad \text{for all}\quad x,y\in X.
\]
\end{defn}
Leader \cite{leader1983equivalent} introduced the following class of contractions in metric spaces, and subsequently termed \emph{Leader contractions} in the literature. 
\begin{defn}[\cite{keeler1969theorem}, \cite{leader1983equivalent} ]Let $(X,\Delta)$ be a $b$-metric space, $T:X\to X$ 
\begin{enumerate}
\item is called a Leader contraction if for each $\epsilon>0$ there exist $\delta>0$, and $r\in \mathbb {N}$ such that:
\[
\Delta(x,y)<\epsilon+\delta \rightarrow  \Delta(T^{r}x,T^{r}y)<\epsilon \quad \text{ for all } x,y\in X.
\]
\item  is called a Meir-Keeler contraction if for each $\epsilon>0$ there exist $\delta>0$ such that:
\[
\epsilon \leq \Delta(x,y)<\epsilon+\delta \rightarrow  \Delta(Tx,Ty)<\epsilon \quad \text{ for all } x,y\in X.
\]
\end{enumerate}
\end{defn}
\section{Main Results}\label{mai-res}
Our main contributions establish the fundamental relationships between contraction classes and prove a key fixed-point theorem in $b$-metric spaces.

\begin{prop}\label{proper}(Proper Containments of Contraction Classes)
Let $\mathrm{Ma}$, $\mathrm{BW}$, $\mathrm{MK}$, $\mathrm{Le}$, and $\mathrm{N.Le}$ denote the classes of Matkowski, Boyd-Wong, Meir-Keeler, Leader, and non-expansive Leader contractions, respectively. Then,    $\mathrm{Ma} \subsetneq \mathrm{N.Le}$,  $\mathrm{MK} \subsetneq \mathrm{N.Le}$,   $\mathrm{N.Le} \subsetneq \mathrm{Le}$.

\end{prop}
\begin{proof}
We have $\mathrm{Ma}, \mathrm{MK} \subset \mathrm{Le}$ (see \cite{jachymski2011equivalent}). Since every Meir-Keeler or Matkowski contraction is non-expansive, we observe that $\mathrm{Ma} \subset \mathrm{N.Le}$ and $\mathrm{MK} \subset \mathrm{N.Le}$. On the other hand, since $\mathrm{Ma}$ and $\mathrm{MK}$ are not comparable, we deduce that each of these inclusions is proper. Example~\ref{leader-exm} exhibits a Leader contraction that is not non-expansive. Consequently, the non-expansive Leader contractions constitute a proper subset of Leader contractions.
\end{proof}
We now establish a fixed point theorem for non-expansive Leader contractions, which form a class that properly contains all contraction types in Jachymski's hierarchy diagram \cite{jachymski2007nonlinear} except for general Leader contractions.
\begin{thm}\label{thm_leader}
Let $(X,\Delta)$ be a complete $b$-metric space with coefficient $s\geq 1$, and let $T:X\to X$ a non-expansive Leader contraction with bounded orbits. Then $T$ has a fixed point. 
\end{thm}
\begin{proof}Let $x\in X$, and $\mathcal{O}(x) =\{T^{n}x\}_{n=0}^{\infty}$ be the orbit of $T$ at $x$. We define $\Sigma$ as ordered paires of $(m,n)$, where $m,n$ are sequences of non-negative integers, as follows: 
\begin{equation}
 \begin{aligned}
    \Sigma = \Big\{ 
        \big( m, n \big) \;\big|\;
           m=\{a_k\}_{k=0}^{\infty}, n=\{b_k\}_{k=0}^{\infty}, &  m(k)\leq n(k)\quad\text{for all }k\in \mathbb{Z}^{+}\\
 & ,\lim_{k\to\infty} m(k) = \lim_{k\to\infty} n(k) = \infty
          \Big\}
  \end{aligned}
\end{equation}
for $(m,n)\in \Sigma, p\in \mathbb{Z}^{+}$, define
\begin{align}\label{sig_inf}
\sigma(m,n,p) = \inf_{k\geq 0}d(T^{m(k)+p}x,T^{n(k)+p}x),
\end{align}
\begin{equation}\label{sig_def}
\sigma(p) = \inf_{(m,n)\in \Sigma}\sigma(m,n,p).
\end{equation}
\begin{equation}\label{theta_sup}
\theta(p) = \sup_{(m,n)\in \Sigma}\sigma(m,n,p).
\end{equation}
Since $T$ has bounded orbits, for all $(m,n)\in \sigma, p\in \mathbb{Z}^{+} $ we have:
\[
0\le \sigma(m,n,p), \sigma(p),\theta(p)<\infty.
\]
 On the other hand, since $T$ is non-expansive, we have
\[
\sigma(p+1)\leq \sigma(p)\quad \text{for all }p\in \mathbb{Z}^{+}.
\]
This shows that there exists $\sigma\geq 0$ such that 
\[
\sigma(p)\to \sigma \text{ as }p\to \infty\quad\text{ ,and }\sigma(p)\geq  \sigma\quad\text{ for all }p\in \mathbb{Z}^{+}.
\]
We claim that $\sigma =0$. If $\sigma>0$, let $\delta>0$ and $r\in \mathbb{Z}^{+}$ such that for all $x,y\in X$:
\begin{equation}\label{leader_delta}
d(x,y)<\sigma+\delta\rightarrow d(T^{r}x,T^{r}y) <  \sigma
\end{equation}
There exists $p_0\in \mathbb{Z}^{+}$ such that
\[
\sigma(p)<\sigma+\delta\quad\text{ for all }p\ge p_0
\]
Let $p\ge p_{0}$, invoning \eqref{sig_def}, there exists $(m,n)\in \Sigma$ such that
\[
\sigma(m,n,p)<\sigma+\delta
\]
by \eqref{sig_inf} we have:
\[
\inf_{k\geq 0}d(T^{m(k)+p}x,T^{n(k)+p}x)<\sigma+\delta.
\]
This implies that there exists $k_0\ge 0$ such that:
\[
d(T^{m(k_0)+p}x,T^{n(k_0)+p}x)<\sigma+\delta.
\]
By \eqref{leader_delta} we have:
\begin{equation}\label{lead_1}
d(T^{m(k_0)+p+r}x,T^{n(k_0)+p+r}x) < \sigma.
\end{equation}
Therefore
\[
\inf_{k\geq 0} d(T^{m(k)+p+r}x,T^{n(k)+p+r}x)< \sigma.
\]
This implies that:
\begin{equation}\label{lead_3}
\sigma(m+r,n+r,p) < \sigma.
\end{equation}
Since $(m+r,n+r)\in \Sigma$,  invoking \eqref{sig_inf} we deduce that:
\begin{equation}
\sigma(p)\leq \sigma(m+r,n+r,p) < \sigma.
\end{equation}
This is a contradiction that  proves our claim that is $\sigma = 0$.\\
Since $T$ is non-expansive for all $(m,n)\in \Sigma$ we have:
\[
\sigma(m,n,p+1)\leq \sigma(m,n,p)
\]
This implies that 
\[
\sup_{(m,n)\in\Sigma}\sigma(m,n,p+1)\leq  \inf_{(m,n)\in\Sigma} \sigma(m,n,p),
\]
therefore:
\[
0\leq \theta(p+1) \leq \sigma(p).
\]
This implies that 
\begin{equation}
\lim_{p\to\infty}\theta(p) = 0
\end{equation}
We claim that the sequence $\{T^{n}x\}$ is Cauchy. If not, there exists $\epsilon>0$ such that $m(k),n(k)\to \infty$ as $k\to\infty$, $m(k)\leq n(k)$ for all $k\geq 0$ such that:
\[
\epsilon\leq d(T^{m(k)}x,T^{n(k)}x)
\]
Let $p\in \mathbb{N}$, without loss of generality we assume that $m(k),n(k)\geq p$ for all $k\geq 0$. We have:
\begin{equation}
\epsilon\leq d(T^{m(k)}x,T^{n(k)}x) =  d(T^{m(k)-p+p}x,T^{n(k)-p+p}x).
\end{equation}
We observe that $(m_1,n_1) = (m-p, n-p)\in \Sigma$, therefore we get:
\begin{align}
\epsilon&\leq \inf_{k\geq 0}d(T^{m(k)-p+p}x,T^{n(k)-p+p}x)\nonumber\\
& = \inf_{k\geq 0}d(T^{m_1(k)+p}x,T^{n_1(k)+p}x)\nonumber\\
&= \sigma(m_1,n_1,p)\leq \theta(p)\nonumber.
\end{align}
Hence, we get
\[
\epsilon\leq \theta(p).
\]
Since $p$ is arbitrary, taking the limit as $p\to \infty$ we get $\epsilon = 0$. This contradiction proves our claim, that is $\{T^{n}x\}$ is a Cauchy sequence. Since $(X,\Delta)$ is a complete metric space it converges to a point $z\in X$. Since $T$ is continuous $Tz=z$ and our proof is complete.
\end{proof}
\begin{rem}
For non-expansive $T$: (1) $\exists$ fixed point $\iff$ (2) $\exists x$ with bounded $\{T^nx\}$ $\iff$ (3) all orbits bounded. Thus in Theorem \ref{thm_leader}, either $\{T^nx\}$ converges $\forall x$ or $T$ is fixed-point-free.
\end{rem}

\begin{rem}
Theorem \ref{thm_leader} remains valid when $\Delta$'s triangle inequality is replaced by: (i) uniqueness of limits, and (ii) preserved contraction properties.
\end{rem}
\section{Examples}\label{exams}
\begin{ex}\label{lu-exm} 
Lu et al. presented a $b$-metric space $(X,\Delta)$,  where $X=\{s_n=\sum_{k=1}^n 1/k : n\in\mathbb{N}\}$, and  define $T: X\to X$ by $Ts_n = s_{n+1}$ for all $n\in\mathbb{N}$ (See Example 2.3 \cite{lu2023question} for the definition of $\Delta$). Lu et al. proved that $T$ is a Meir-Keeler mapping on the $b$-metric space $(X,\Delta)$ and has no fixed point. We claim that $T$ does not satisfy the conditions of Theorem \ref{thm_leader}, so it does not present a contradiction to our result.\\
 We observe that $T$ satisfies the following condition for all $x,y\in X$ (See Example 2.3 in \cite{lu2023question}, Relation (2.10)).
\begin{equation}\label{bounded_1}
 |x-y|/2 \leq \Delta(x,y) \leq |x-y|,
\end{equation}
where $| x-y|$ is the absolute value of $x-y$ for all $x,y\in X$.  The orbit of $T$ based at $x=1$ is 
\[
\mathcal{O}(1)=\{\Sigma_{i = 1}^{n}{1/i} \}_{n=1}^{\infty}= X. 
\]
By \eqref{bounded_1} $X$ is unbounded. This shows that $T$ has an unbounded orbit and hence does not sarisfies the conditions of Theorem \ref{thm_leader}.
\end{ex}
\begin{ex}\label{leader-exm}
Consider the mapping $T\colon [0,\frac{3}{4}] \to [0,\frac{3}{4}]$ defined by
\[
T(x) = 
\begin{cases} 
\frac{x}{3} & \text{if } x \in \big[0, \frac{1}{2}\big], \\ 
\frac{x}{3} + \frac{1}{4} & \text{if } x \in \big(\frac{1}{2}, \frac{3}{4}\big],
\end{cases}
\]
$T$ is a Leader contraction but not a non-expansive one.

\noindent\textbf{Leader Contraction Verification:}

For $n \geq 3$ and any $x,y \in [0,\frac{3}{4}]$, the iterates satisfy:
\begin{itemize}
\item If $x,y \in [0,\frac{1}{2}]$: $|T^nx - T^ny| = \frac{|x-y|}{3^n}$
\item If $x,y \in (\frac{1}{2},\frac{3}{4}]$: $|T^nx - T^ny| = \frac{|x-y|}{3^n}$
\item For $x \in [0,\frac{1}{2}]$, $y \in (\frac{1}{2},\frac{3}{4}]$:
\[
|T^nx - T^ny| \leq \frac{|x-y|}{3^n} + \frac{1}{4\cdot 3^{n-1}}
\]
\end{itemize}

To verify the Leader condition, for any $\varepsilon > 0$:
\begin{enumerate}
\item Choose $r \in \mathbb{N}$ such that $\frac{1}{4\cdot 3^{r-1}} < \frac{\varepsilon}{2}$
\item Take $\delta =( \frac{3^r}{2}-1)\cdot \varepsilon$.
\end{enumerate}
Then $\Delta(x,y) < \varepsilon + \delta$ implies:
\[
\Delta(T^rx, T^ry) \leq \frac{\varepsilon + \delta}{3^r} + \frac{\varepsilon}{2} < \varepsilon
\]

\noindent\textbf{Non-Expansiveness Failure:}

Consider $x=\frac{1}{2}$ and $y_k=\frac{1}{2}+\frac{1}{k}$ for $k\in\mathbb{N}$:
\[
\lim_{k\to\infty} \Delta(Tx,Ty_k) = \left|\frac{1}{6}-\left(\frac{1}{6}+\frac{1}{4}\right)\right| = \frac{1}{4}
\]
while $\Delta(x,y_k) = \frac{1}{k} \to 0$. This shows $T$ cannot be non-expansive.
\end{ex} 

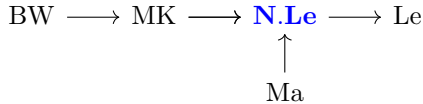
\begin{figure}[h]
\centering
\begin{tikzcd}[row sep=1.5em,column sep=2em]
\mathrm{BW}\arrow[r]& \mathrm{MK} \arrow[r]  \arrow[r]&  \mathrm{\bf \color{blue} N.Le} \arrow[r] & \mathrm{Le}  \\
&&  \mathrm{Ma} \arrow[u] \\
\end{tikzcd}
\caption{Hierarchy of contraction classes in $b$-metric spaces. Blue highlights our new class of non-expansive Leader contractions ($\mathrm{N.Le}$), showing its position between Meir-Keeler ($\mathrm{MK}$) and Leader ($\mathrm{Le}$) contractions, while properly containing Matkowski contractions ($\mathrm{Ma}$). The arrows indicate proper inclusions.}
\label{nle:diagram}
\end{figure}
\section{Conclusion}

We have established a complete fixed-point theory for non-expansive Leader contractions in $b$-metric spaces, resolving the limitation exposed by Lu et al.'s counterexample. Our main contributions are:

\begin{itemize}
\item The introduction of non-expansive Leader contractions ($N.Le$) as a proper subclass of Leader contractions that properly contains both Matkowski and Meir-Keeler contractions (Proposition \ref{proper})

\item A fixed-point theorem (Theorem \ref{thm_leader}) showing that bounded orbits guarantee existence for these contractions

\item The complete hierarchy diagram (Figure \ref{nle:diagram}) demonstrating how non-expansive Leader contractions unify the contraction landscape in $b$-metric spaces
\end{itemize}

The key innovation lies in identifying the precise conditions (non-expansiveness and bounded orbits) that allow fixed-point results while avoiding the pathological cases. This work suggests several natural extensions:

\begin{itemize}
\item Investigation of other function classes that may satisfy similar conditions
\item Applications to integral equations and fractional calculus where $b$-metric spaces naturally appear
\item Extension to partial $b$-metric spaces and other generalized metric structures
\end{itemize}

Our results restore the elegant hierarchy of contraction types in the $b$-metric setting, providing a foundation for further developments in fixed-point theory.
\newpage
\vspace{0.5cm}
\section*{Declaration of Competing Interest}
\noindent The author declares that there are no competing interests.
\section*{Declaration of generative AI and AI-assisted technologies in the writing process}
During the preparation of this work, the author used DeepSeek for language refinement and technical suggestions. The AI-assisted content was rigorously reviewed, edited, and verified by the author, who assumes full responsibility for the accuracy, integrity, and originality of the final publication.\\
\section*{Funding}
The author affirms that the preparation of this manuscript was not supported by any funding, grants, or other financial assistance.  


\subsection*{Acknowledgment}
Many thanks to our \TeX-pert for developing this class file.

\end{document}